\documentclass{amsart}
\usepackage{amsthm,amscd,amssymb,verbatim,epsf,amsmath,eurosym,amsfonts,mathrsfs,graphicx}
\usepackage[colorlinks=true,linkcolor=blue,citecolor=blue]{hyperref}
\begin{document}
\theoremstyle{plain}
\newtheorem{Thm}{Theorem}
\newtheorem{Cor}[Thm]{Corollary}
\newtheorem{Ex}[Thm]{Example}
\newtheorem{Con}[Thm]{Conjecture}
\newtheorem{Main}{Main Theorem}
\newtheorem{Lem}[Thm]{Lemma}
\newtheorem{Prop}[Thm]{Proposition}

\theoremstyle{definition}
\newtheorem{Def}[Thm]{Definition}
\newtheorem{Note}[Thm]{Note}
\newtheorem{Question}[Thm]{Question}

\theoremstyle{remark}
\newtheorem{notation}[Thm]{Notation}
\renewcommand{\thenotation}{}

\errorcontextlines=0
\numberwithin{equation}{section}
\numberwithin{Thm}{section}

\title[Extension of Asgeirsson's Theorem]%
   {An Extension of Asgeirsson's Mean Value Theorem for Solutions of the ultra-hyperbolic Equation \\
   in Dimension Four}
\author{Guillem Cobos}\address{
          School of STEM\\
          Institute of Technology, Tralee \\
          Clash \\
          Tralee  \\
          Co. Kerry \\
          Ireland.}
\email{guillem.cobos93@gmail.com}
\author{Brendan Guilfoyle}
\email{brendan.guilfoyle@ittralee.ie}

\begin{abstract} In 1937 Asgeirsson established a mean value property for solutions of the general ultra-hyperbolic equation in $2n$ variables. In the case of four variables, it states that the integrals of a solution over certain pairs of conjugate circles are the same. In this paper we extend this result to non-degenerate conjugate conics, which include the original case of conjugate circles and adds the  new case of conjugate hyperbolae. 

The broader context of this result is the geometrization of Fritz John's 1938 analysis of the ultra-hyperbolic equation. Solutions of the equation arise as the compatibility for functions on line space to come from line integrals of functions in Euclidean 3-space. The introduction of the canonical neutral Kaehler metric on the space of oriented lines clarifies the relationship and broadens the paradigm to allow new insights.

In particular, it is proven that a solution of the ultra-hyperbolic equation has the mean value property over any pair of curves that arise as the image of John's conjugate circles under a conformal map. These pairs of curves are then shown to be conjugate conics, which include circles and hyperbolae. 

John identified conjugate circles with the two rulings of a hyperboloid of 1-sheet. Conjugate hyperbolae are identified with the two rulings of either a piece of a hyperboloid of 1-sheet or a hyperbolic paraboloid.
\end{abstract}

\vspace*{-5mm}

\maketitle
\tableofcontents

\section{Introduction and Results}

A twice continuously differentiable function $u$ of four variables is said to be a solution of the ultra-hyperbolic equation if it satisfies

\begin{equation}\label{UHE}
\frac{\partial^2 u}{\partial x_1^2} + \frac{\partial^2 u}{\partial x_2^2} -
\frac{\partial^2 u}{\partial x_3^2} -
\frac{\partial^2 u}{\partial x_4^2} = 0.
\end{equation}
In 1937 Asgeirsson's proved \cite{LA} that a solution of the ultra-hyperbolic equation satisfies 

\begin{equation} \label{asgeirsson}
\int_0^{2\pi} u(a + r\cos\theta, b+ r\sin\theta, c, d) d\theta
= 
\int_0^{2\pi} u(a, b,c+r\cos\theta, d+r\sin\theta) d\theta,
\end{equation}
for all $a,b,c,d,r\in{\mathbb R}$ with $r\geq0$. This mean value theorem was subsequently considered by Fritz John \cite{fjohn} and discussed at length by Courant and Hilbert \cite{couranthilbert}. Indeed, John showed that the mean value property is sufficient for a function to satisfy the ultra-hyperbolic equation, thus allowing one to define generalised solutions of a second order partial differential equation that were merely continuous - see the discussion on page 301 of, and Theorem 1.1 of \cite{fjohn}.

Moreover, John states in passing that the integrals can be taken over "conjugate conics" if a "certain variable of integration is used". He notes that the case of conjugate hyperbolae has not been considered, and that it is not directly amenable to the methods of his proof, which utilizes affine maps  (p306 of \cite{fjohn}). In this paper we fill this gap by utilizing the conformal geometry of line space and thus extending the method from affine maps to projective maps.

From a geometric perspective, Asgeirsson's result can be viewed in ${\mathbb R}^4$ endowed with the flat indefinite metric of signature (2,2):
\[
g=dx_1^2+dx_2^2-dx_3^2-dx_4^2.
\]
Thus, the second order ultra-hyperbolic operator is the Laplacian of $g$ and solutions of the ultra-hyperbolic equation are the harmonic functions. Moreover, the two curves over which the mean value is computed in (\ref{asgeirsson}) are circles of center $(a,b,c,d)$ and radius $r$, one of them lying in the $x_1x_2-$plane and the other in the $x_3x_4-$plane. We refer to these as conjugate circles and throughout we denote the pair $({\mathbb R}^4,g)$ simply by ${\mathbb R}^{2,2}$.

Note that the induced metric on one plane is positive definite, on the other it is negative definite and that the planes are mutually orthogonal. Moreover, both curves are round circles in the metric induced on the planes by $g$.

In this paper Asgeirsson's theorem is extended to more general pairs of curves, called non-degenerate conjugate conics. Our main result can be stated as follows:

\vspace{0.1in}
\begin{Thm}\label{t:1} 

Let $u$ be a solution of the ultra-hyperbolic equation (\ref{UHE}). Let $S,S^\perp$ be a pair of non-degenerate conjugate conics in ${\mathbb R}^{2,2}$. 

Then
\[
\int_S u \ dl = \int_{S^\perp} u \ dl,
\]
where $dl$ is the line element induced on the curves $S,S^\perp$ by the flat metric $g$.
\end{Thm}
\vspace{0.1in}

To define the pairs of curves, introduce the following nomenclature. Consider a pair of orthogonal planes $\pi, \pi^\perp$ through a point $O$ in ${\mathbb R}^{2,2}$, where orthogonality is with respect to the neutral metric on $\mathbb{R}^{2,2}$. If either plane is totally null (the induced metric is identically zero on such a plane) then $\pi=\pi^\perp$. The induced metric on $\pi$ can be definite or indefinite but is assumed non-degenerate. Thus we exclude all degenerate cases in what follows - these cases will be dealt with in a separate paper \cite{Cogui2}.

A pair of {\em non-degenerate conjugate conics} consists of a pseudo-circle $S$ of center $O$ and  square-radius $c^2$ in $\pi$; and a pseudo-circle $S^\perp$ of center $O$ and square-radius $-c^2$ in $\pi^\perp$, for any pair of non-degenerate orthogonal planes $\pi, \pi^\perp$.

A pair of non-degenerate conjugate conics can be either:
\begin{enumerate} 
\item[(i)]  a pair of {\em circles} with a common center and opposite square-radius in orthogonal definite planes, or 
\item[(ii)] a pair of {\em hyperbolae} with a common center and opposite square-radius in orthogonal indefinite planes.
\end {enumerate} 

\vspace{0.1in}

Theorem \ref{t:1} yields new mean value theorems for solutions of equation (\ref{UHE}). For example,
\begin{equation}\label{e:hyp}
\int_{-\infty}^{\infty} u(a + r\cosh\theta, b, c+ r\sinh\theta, d) d\theta
= 
\int_{-\infty}^{\infty} u(a, b+r\cosh\theta,c, d+r\sinh\theta) d\theta,
\end{equation}
integrating over a pair of {\em conjugate hyperbolae}. As the integrals are now over unbounded domains one also should assume that $u$ is integrable.

This result can be recast in a more geometric way as follows. In \cite{fjohn} the first link was made between the ultra-hyperbolic equation and straight lines in Euclidean 3-space. In particular, ${\mathbb R}^{2,2}$ is identified locally with the space of straight lines and then solutions of the equation (\ref{UHE}) arise from line integrals of functions on ${\mathbb R}^3$.

The space of all oriented lines is diffeomorphic to $TS^2$ and is endowed with a canonical metric ${\mathbb G}$ of signature (2,2) \cite{kahlermetric}. The metric is conformally flat, and so can be conformally mapped to ${\mathbb R}^{2,2}$ locally. Thus solutions of the ultra-hyperbolic equation give rise to solutions of the Laplace equation. Moreover, conjugate conics can be defined in $TS^2$ as discussed in detail in Section \ref{s:2.3}. 

This leads to a more geometric statement of our result:  

\vspace{0.1in}
\begin{Thm} \label{t:2}
Let $v:TS^2\rightarrow \mathbb R$ be a solution of the Laplace equation of the canonical metric $\Delta_{\mathbb G} v=0$. Let $S,S^\perp$ be a pair of non-degenerate conjugate conics in $TS^2$. 

Then the following integral equation holds
$$
\int_{S} v \ d\tau = \int_{S^\perp} v \ d\tau,
$$
where $d\tau$ is the line element induced by ${\mathbb G}$.
\end{Thm}
\vspace{0.1in}

In Section \ref{extended} the background to Asgeirsson's theorem and Fritz John's contributions are discussed. Details are given of the geometric under-pinning of our approach, as well as a proof that Theorems \ref{t:1} and \ref{t:2} are equivalent. 

The picture that emerges is that the splitting of ${\mathbb R}^{2,2}$ into conjugate planes in the original formulation with positive and negative definite metrics, extends to the case where the conjugate planes are both hyperbolic. The curves change from circles to pseudo-circles, but the integral identity remains the same.

In ${\mathbb R}^3$ the picture is richer. The image of a plane in ${\mathbb R}^{2,2}$ under the conformal map to $TS^2$, is a 2-parameter family of lines which twist around each other so that there is no orthogonal surface in ${\mathbb R}^3$. This line congruence enjoys the property that the images of (pseudo) circles are special ruled surfaces. John showed that when the two planes are definite, the conics are the two rulings of a 1-sheeted hyperboloid. We show that

\vspace{0.1in}
\begin{Thm}\label{t:3}
A pair of non-degenerate conjugate conics can be identified with a pair of rulings of the same surface in ${\mathbb{R}}^3$.

In the definite case, they are the two rulings of a 1-sheeted hyperboloid, in the hyperbolic case, they are either part of the two rulings of a 1-sheeted hyperboloid, or the two rulings of a hyperbolic paraboloid. 
\end{Thm}
\vspace{0.1in}

The uniqueness of this situation is well-expressed by the classical result that states that, other than the plane, the only surfaces in ${\mathbb R}^3$ that have more than one ruling by straight lines are the 1-sheeted hyperboloid and the hyperbolic paraboloid \cite{hcv}. 

An outline of the proof of Theorem \ref{t:1} contained in this paper is as follows. For convenience, let $S_0$ and $S_0^\perp$ be the unit circles in the conjugate planes ${\mathbb R}^2\times(0,0)\subset{\mathbb R}^{2,2}$ and $(0,0)\times{\mathbb R}^2\subset{\mathbb R}^{2,2}$ respectively.

In Section \ref{extended} it is shown that a solution of the ultra-hyperbolic equation has the mean value property over any pair of curves $S,S^\perp$ that arise as the image of $S_0, S_0^\perp$ under an arbitrary conformal mapping of ${\mathbb R}^{2,2}$. We refer to this as the conformal extension of Asgeirsson's Theorem which can be expressed in oriented line space as:

\vspace{0.1in}
\begin{Thm} \label{t:4}
Let $v:TS^2\rightarrow \mathbb R$ be a solution of $\Delta_{\mathbb G} v=0$ and $f:TS^2 \rightarrow TS^2$ be a conformal map. Then
\[
\int_{f(S_0)} v \ d\tau = \int_{f(S_0^\perp)} v \ d\tau.
\]
where $d\tau$ is the line element induced by the canonical metric ${\mathbb G}$.
\end{Thm}

\vspace{0.1in}

Section \ref{conformalimage} shows that any pair of non-degenerate conjugate conics comes as $f(S_0), f(S_0^\perp)$ for some conformal map $f$. This completes the proof of Theorem \ref{t:1}. 

 Finally, there are two appendices, one dedicated to numerical checks of the extended Asgeirsson's theorem for a particular solution of the ultra-hyperbolic equation and for specific choices of conjugate conics, and one giving solutions of the ultra-hyperbolic equation generated by certain 3-dimensional step functions.

\vspace{0.1in}
\section{Conformal Extension of Asgeirsson's Theorem}\label{extended}

\subsection{Background}
The X-ray transform is defined by taking the integral of a scalar function over (affine) lines of ${\mathbb R}^n$. That is, given a real function $f:{\mathbb R}^n\rightarrow{\mathbb R}$ and a line $\gamma$ in ${\mathbb R}^n$, define the map 
\[
\chi_f(\gamma) = \int_\gamma fdt,
\]
where $dt$ is the flat line element in ${\mathbb R}^n$.

For functions $f$ with appropriate behaviour at infinity (for example, compactly supported), one obtains the {\em X-ray Transform of f}, namely $\chi_f:{\mathbb L}({\mathbb R}^n)\rightarrow{\mathbb R}:\gamma\mapsto \chi_f(\gamma)$, where ${\mathbb L}({\mathbb R}^n)$ is the space of lines in ${\mathbb R}^n$.

In comparison, the {\em Radon Transform} takes a function and integrates it over hypersurfaces in ${\mathbb R}^n$. By elementary considerations, the space of affine hypersurfaces in ${\mathbb R}^n$ is seen to be $n$-dimensional, the dimension of the underlying space, while the space of affine lines of ${\mathbb R}^n$ has dimension $2n-2$, which is bigger than the dimension of the underlying space for $n>2$. 

Thus, by dimension count, if we consider the problem of inverting the two transforms, given a function on hypersurfaces one can reconstruct the original function on ${\mathbb R}^n$, while, for $n>2$, the problem is over-determined for functions on lines.

In particular, for $n=3$, the space of lines is 4-dimensional and the ultra-hyperbolic equation is the condition that must be satisfied for a function on line space to come from an integral of a function on ${\mathbb R}^3$.  This was shown in 1938 when Fritz John \cite{fjohn} characterised solutions of equation (\ref{UHE})  with certain smoothness and decay conditions as the range of the X-ray transform. 

This impressive result has had a huge impact on fields such as tomography, where information of the interior of a 3D body is deduced from attenuation data of rays traversing the body. The ultra-hyperbolic equation, or an equivalent system of second order partial differential equations known as {\em John's equations}, represent a compatibility condition that tomographic data must satisfy, and so it has been exploited practically to obtain image reconstruction algorithms \cite{conebeam} \cite{fourierrebinning} \cite{defrisetliu} \cite{lichen}. 

Asgeirsson's theorem for the ultrahyperbolic equation has also seen applications in imaging and analysis of lightfields. T. Georgiev, H. Qin and H. Li recently used it to derive new kernels for processing lightfields which provide nice features like accurate depth estimation of planar images \cite{johntransform} \cite{lightfieldcoordinates}. 

More generally, one can consider geodesics on general Riemannian manifolds \cite{Uhlmann1} \cite{Uhlmann2} and an active line of work has been finding the range of the X-ray transform of rank $m$ symmetric tensor fields \cite{denisiuk} \cite{venkateswaran} \cite{nadirashvili}.  

The foundation of John's work lies in Asgeirsson's mean value theorem for solutions of second order linear partial differential equations with constant coefficients (see \cite{LA} and \cite{couranthilbert}). While Asgeirsson's original result holds for solutions of the  $(n,n)$-dimensional ultra-hyperbolic equation, John exploited the geometry of the $(2,2)$ case to obtain it in a more general form. 

The link between line space and the ultra-hyperbolic equation is made concrete by Pl\"ucker coordinates. That is, to describe a line in ${\mathbb R}^3$ John chooses two distinct points $(s_1,s_2,s_3)$ and $(t_1,t_2,t_3)$ on the line and forms the sextet $(p_1,p_2,p_3,q_1,q_2,q_3)$ by
\[
p_1=s_2t_3-t_2s_3 \qquad p_2=s_3t_1-t_3s_1 \qquad p_3=s_1t_2-t_1s_2
\]
\[
q_1=s_1-t_1 \qquad q_2=s_2-t_2 \qquad q_3=s_3-t_3.
\]
Flat coordinates on ${\mathbb R}^{2,2}$ are related to the Pl\"ucker coordinates away from $q_3=0$ by
\[
x_1=\frac{p_2+q_2}{q_3} \quad x_2=\frac{-p_1-q_1}{q_3} \quad x_3=\frac{p_2-q_2}{q_3} \quad x_4=\frac{-p_1+q_1}{q_3}.
\]

The freedom to choose different representative points on a given line means that the above description has a 2 dimensional redundancy in $(p_1,p_2,p_3,q_1,q_2,q_3)$. One dimension is generated by scaling, the other by the fact that the embedding satisfies the quadric equation
\[
p_1q_1+p_2q_2+p_3q_3=0.
\]
John remarks that the deeper connection between the ultra-hyperbolic equation and line space is given by this embedding of line space in the quadric hypersurface in projective 5-space. In what follows we will explore this further from the point of view of conformal differential geometry. 

\vspace{0.1in}
\subsection{Geometrization of line space}

To remove the extra degrees of freedom inherent in Pl\"ucker coordinates, proceed as in \cite{kahlermetric} and identify the space of oriented lines in ${\mathbb R}^3$ with $TS^2$, the total space of the tangent bundle of the 2-sphere. This double covers the space of unoriented lines, but this makes no difference in what follows as all our calculations are local.

Starting with the usual holomorphic coordinate $\xi$ on $S^2$ obtained by stereographic projection, construct complex coordinates $(\xi,\eta)$ on $TS^2$ by identifying 
\[
(\xi,\eta)\leftrightarrow \eta\frac{\partial}{\partial\xi}+\bar{\eta}\frac{\partial}{\partial\bar{\xi}}\in T_\xi S^2.
\]
Thus $\xi\in S^2$ is the direction of the oriented line and the complex number $\eta$ measures the displacement of the line from the origin.

\vspace{0.1in}
\begin{Thm}\cite{kahlermetric}\label{t:gk}
The space of oriented lines $TS^2$ admits a canonical metric ${\mathbb G}$ that is invariant under the Euclidean group acting on lines. The metric is of neutral signature (2,2), is conformally flat and scalar flat, but not Einstein. It can be supplemented by a complex structure ${\mathbb J}$ and symplectic structure $\omega$, so that $(TS^2,{\mathbb G},{\mathbb J},\omega)$ is a neutral K\"ahler 4-manifold.
\end{Thm}
\vspace{0.1in}
In terms of the local coordinates $(\xi,\eta)$ the neutral metric is
\begin{equation}\label{e:metric}
ds^2=4(1+\xi\bar{\xi})^{-2}{\mathbb{I}}\mbox{m}\left(d\bar{\eta} d\xi+\frac{2\bar{\xi}\eta}{1+\xi\bar{\xi}}d\xi d\bar{\xi}\right).
\end{equation}
Recall that a (pseudo-) Riemannian manifold $(M,{\mathbb G})$ is {\em conformally flat} if it can be covered by local coordinate systems $\{x_j\}_1^n$ such that the metric is $ds^2=\Omega^2\sum_j \pm dx^2_j$, $\Omega$ being a non-zero function defined on the coordinate neighbourhood.

The following result supplies local conformal coordinates for line space:
\vspace{0.1in}
\begin{Prop}\label{p:conf}
For complex coordinates $(\xi,\eta)$ on $TS^2$, over the upper hemisphere $|\xi|^2<1$ the conformal coordinates $(Z_1=x_1+ix_2,Z_2=x_3+ix_4)$ are
\begin{equation}\label{e:confco1}
Z_1=x_1+ix_2=\frac{2}{1-\xi^2\bar{\xi}^2}\left(\eta+\xi^2\bar{\eta}-i(1+\xi\bar{\xi})\xi\right),
\end{equation}
\begin{equation}\label{e:confco2}
Z_2=x_3+ix_4=\frac{2}{1-\xi^2\bar{\xi}^2}\left(\eta+\xi^2\bar{\eta}+i(1+\xi\bar{\xi})\xi\right),
\end{equation}
with inverse
\begin{equation}\label{e:confcoinv1}
    \xi=\frac{i(Z_1-Z_2)}{2-\sqrt{4+|Z_1-Z_2|^2}},
\end{equation}
\begin{equation}\label{e:confcoinv2}
    \eta=\frac{Z_1+Z_2}{2-\sqrt{4+|Z_1-Z_2|^2}}+\frac{(Z_1-Z_2)(|Z_1|^2-|Z_2|^2)}{2(2-\sqrt{4+|Z_1-Z_2|^2})^2}.
\end{equation}
\end{Prop}
\begin{proof}
The result follows from pulling back the metric (\ref{e:metric}) by the transformation (\ref{e:confcoinv1}) and (\ref{e:confcoinv2}) and the result is
\[
ds^2=\frac{1}{1+{\textstyle{\frac{1}{4}}}|Z_1-Z_2|^2}\left(dZ_1d\bar{Z}_1-dZ_2d\bar{Z}_2\right).
\]
\end{proof}
\vspace{0.1in}
Note that these coordinates are only local and exclude lines whose direction is parallel to the $xy$-plane in ${\mathbb R}^3$ ($q_3=0$ or $|\xi|=1$). This comes from our original choice of holomorphic coordinate $\xi$ on $S^2$ via stereographic projection from the south pole. The whole of $TS^2$ can be covered by such coordinate patches, each of which is diffeomorphic to the product of an open hemisphere and ${\mathbb R}^2$. 

As John never refers to a metric on the space of lines, many of his computations involve the appearance of a mysterious factor - see e.g. equations (6) or (7), or Theorem 2.2 of \cite{fjohn}. 

We will now demonstrate that the origin of these factors is precisely the conformal factor in this canonical metric. First recall some elementary facts about conformal maps of scalar flat pseudo-Riemannian manifolds.

Let $(M,g^M)$ and $(N, g^N)$ be pseudo-Riemmanian manifolds of dimension $n$ with vanishing scalar curvatures. 

\vspace{0.1in}
\begin{Prop}\label{cfharmonic}
Let $f:(M,g^M)\rightarrow (N,g^N)$ be a conformal map with conformal factor $\Omega:M\rightarrow \mathbb R$, so that $f^*g^N=\Omega^2 g^M$.

Then $\Omega^{\frac{n-2}{2}}$ is a solution of Laplace's equation on $M$, that is
$$
\Delta_{g^M}\Omega^{\frac{n-2}{2}}:= g^{ij}\nabla_i\nabla_j \Omega^{\frac{n-2}{2}} =0,
$$ 
where $\nabla$ is the Levi-Civita connection associated with $g^M$.
\end{Prop}
\begin{proof}
Denote the pulled back metric by $\tilde g^N=f^*g^N$ so that $\tilde g^N=\Omega^2 g^M$. 

The change of the scalar curvature on $M$ under a conformal change of the metric is given by the well-known relation \cite{aubin}

\begin{equation}\label{scalarconformal}
\tilde R = \Omega^{-2} \left(R + \frac{4(n-1)}{n-2}\Omega^{-\frac{n-2}{2}} \Delta_{g^M}\Omega^{\frac{n-2}{2}}\right).
\end{equation}

Since the scalar curvatures of both $g^M$ and $g^N$ are assumed to be zero, $R = \tilde R = 0$, we conclude that $\Delta_{g^M} \Omega^{\frac{n-2}{2}} = 0$. 

\end{proof}
\vspace{0.1in}

Conformal maps of scalar flat metrics also preserve solutions of the Laplace equation as follows:

\vspace{0.1in}
\begin{Prop}\label{conformalsolutions}
Let $(M,g^M)$ and $(N, g^N)$ be two scalar flat pseudo-Riemannian manifolds of dimension $n$ and $f:M \rightarrow N$ be a conformal map with conformal factor $\Omega$. 

Then, for any $u:N \rightarrow \mathbb R$, u is a solution of $\Delta_{g^N} u=0$ iff $v=\Omega^{\scriptstyle{\frac{n-2}{2}}}u\circ f$ is a solution of $\Delta_{g^M}v=0$.

\end{Prop}
\begin{proof}
As before let $\tilde g^N$ be the pullback metric of $g^N$ under $f$. From the conformallity  of $f$ we have $\tilde g^N = \Omega^2 g^M$ and by the standard conformal change of the Laplace operator (see e.g. \cite{aubin})
\[
0 = \Delta_{\tilde{g}^N} u= \Omega^{-\frac{n+2}{2}}\left(
\Delta_{g^M} \Omega^{\frac{n-2}{2}} (u\circ f) - (u\circ f)\Delta_{g^M} \Omega^{\frac{n-2}{2}}
\right).
\]

By Proposition \ref{cfharmonic} the second term vanishes and so we are left with $\Delta_{g^M} \Omega^{\frac{n-2}{2}} (u\circ f) =0$.

\end{proof}
\vspace{0.1in}

Apply this to the neutral 4-manifold $(TS^2,{\mathbb G})$ of oriented lines in ${\mathbb R}^3$. 
By Theorem \ref{t:gk} and Proposition \ref{p:conf} the metric is locally conformal to the flat metric with conformal factor
\[
\Omega=\frac{1}{(1+{\textstyle{\frac{1}{4}}}|Z_1-Z_2|^2)^{\scriptstyle{\frac{1}{2}}}}=\frac{1+\xi\bar{\xi}}{1-\xi\bar{\xi}}.
\]
Denote this local conformal map by $f:{\mathbb R}^{2,2}\rightarrow TS^2$.

Applying the previous Proposition we have
\vspace{0.1in}
\begin{Cor}
Let $u:{\mathbb R}^{2,2}\rightarrow{\mathbb R}$ and $v:TS^2\rightarrow{\mathbb R}$ be related by $v\circ f=\Omega u$.

Then $u$ is a solution of the ultra-hyperbolic equation (\ref{UHE}) iff $\Delta_{\mathbb G} v=0$.
\end{Cor}
\vspace{0.1in}

In fact this can be seen directly by writing the Laplacian acting on functions $v:TS^2\rightarrow{\mathbb R}$ in the two sets of coordinates:
\begin{align}
    \Delta_{\mathbb G}v=&g^{jk}\nabla_j\nabla_k v \nonumber\\
    &=i(1+\xi\bar{\xi})^2\left(\frac{\partial^2}{\partial\xi\partial\bar{\eta}}-\frac{\partial^2}{\partial\bar{\xi}\partial\eta}-\frac{2(\xi\bar{\eta}-\bar{\xi}\eta)}{1+\xi\bar{\xi}}\frac{\partial^2}{\partial\eta\partial\bar{\eta}}\right)v\nonumber\\
    &=(1+{\textstyle{\frac{1}{4}}}|Z_1-Z_2|^2)^{\scriptstyle{\frac{3}{2}}}\left(\frac{\partial^2}{\partial Z_1 \partial\bar{Z}_1}-\frac{\partial^2}{\partial Z_2\partial\bar{Z}_2}\right)\frac{v}{(1+{\textstyle{\frac{1}{4}}}|Z_1-Z_2|^2)^{\scriptstyle{\frac{1}{2}}}}.\nonumber
\end{align}

\vspace{0.1in}
Note that the conformal factor $\Omega$ is precisely the mysterious factor on the right hand side of equation (6) of John's paper \cite{fjohn}.  

We can now see that Theorem \ref{t:2} follows from Theorem \ref{t:1} with $u=\Omega v$ as
\[
\int_Sudl=\int_S\Omega v dl=\int_S vd\tau,
\]
the line elements of ${\mathbb G}$ and $g$ being related by $d\tau=\Omega dl$. The conformal factor in the line element appears in Theorem 2.2 of \cite{fjohn} in an ad hoc manner as the cosine of the angle formed by the ruling of the general 1-sheeted hyperboloid with its axis. Here we gain an insight into its geometric significance as the line element of the neutral metric.

\vspace{0.1in}
\subsection{Conformal planes}\label{s:2.3}

This section explores conjugate planes and their geometry in line space. In particular, we have defined conjugate planes $\pi,\pi^\perp$ in ${\mathbb R}^{2,2}$, and want to consider their image in oriented line space under the conformal maps (\ref{e:confco1}) and (\ref{e:confco2}). 

\vspace{0.1in}
\begin{Def}
A  {\it conformal plane} is a surface $\pi\subset TS^2$ given by a linear equation in conformal coordinates (\ref{e:confco1}) and (\ref{e:confco2}):
\begin{equation}\label{e:linplane}
\alpha_1Z_1+\beta_1\bar{Z}_1+\alpha_2Z_2+\beta_2\bar{Z}_2=\gamma.
\end{equation}
for $\alpha_1,\alpha_2,\beta_1,\beta_2,\gamma\in{\mathbb C}$.
\end{Def}
\vspace{0.1in}

For brevity we denote both the plane in ${\mathbb R}^{2,2}$ and its image in $TS^2$ by $\pi$.  By a translation in ${\mathbb R}^{2,2}$ (which corresponds to a translation and rotation in ${\mathbb R}^3$) set $\gamma$ to zero. Thus $\pi$ contains the vertical line through the origin in ${\mathbb R}^3$ $(\xi,\eta)=(0,0)$.

A conformal plane can be one of two types. Consider the canonical projection $TS^2\rightarrow S^2$ which takes an oriented line to its direction, in coordinates $(\xi,\eta)\mapsto \xi$. 

\vspace{0.1in}
\begin{Def}
A  conformal plane $\pi\subset TS^2$ is {\em graphical} if its projection onto $S^2$ has maximal rank (real dimension two). A  conformal plane is {\em non-graphical} if its projection onto $S^2$ has rank one or zero.   
\end{Def}
\vspace{0.1in}
In the case of rank zero, the conformal plane is the whole fibre of the bundle, which is totally null, and since such a plane is self-orthogonal, in what follows we will ignore this case.
\vspace{0.1in}

\begin{Prop}
A conformal plane is graphical iff $|\alpha_1+\alpha_2|^2-|\beta_1+\beta_2|^2\neq0$.

A graphical conformal plane through the origin is given in graphical coordinates $(\xi,\eta=F(\xi,\bar{\xi}))$ for $|\xi|<1$ by
\begin{equation}\label{e:sect}
\eta=\frac{1}{1-\xi\bar{\xi}}\left(\alpha\xi+\beta\bar{\xi}-\bar{\alpha}\xi^2\bar{\xi}-\bar{\beta}\xi^3\right),
\end{equation}
where,
\[
\alpha=\frac{i(-(\alpha_1-\alpha_2)(\bar{\alpha}_1+\bar{\alpha}_2)+(\beta_1+\beta_2)(\bar{\beta}_1-\bar{\beta}_2)}{-|\alpha_1+\alpha_2|^2+|\beta_1+\beta_2|^2},
\]
and
\[
\beta=-\frac{2i(\bar{\alpha}_1\beta_2-\bar{\alpha}_2\beta_1)}{-|\alpha_1+\alpha_2|^2+|\beta_1+\beta_2|^2}.
\]
A non-graphical conformal plane through the origin is given in coordinates $(u,v)$  by 
\begin{equation}\label{e:nongcp}
\xi=uie^{{\scriptstyle{\frac{i}{2}}}(\theta+\phi)} \qquad\qquad
\eta=\frac{2Hiv(1-u^2e^{2i\theta})}{1-u^4}e^{-{\scriptstyle{\frac{i}{2}}}(\theta-\phi)}
\end{equation}
where $(u,v)$ are parameters on the conformal plane with $u\in[0,1),v\in{\mathbb R}$ and $\theta,\phi\in [0,\pi]$, $H\in{\mathbb R}_+$ are such that
\[
\alpha_1=1 \qquad \beta_1=0 \qquad \alpha_2=He^{i\theta}-1 \qquad \beta_2=He^{i\phi}.
\]
\end{Prop}
\begin{proof}
Equation (\ref{e:sect}) follows from rearranging the result of substituting equations (\ref{e:confco1}) and (\ref{e:confco2}) in equation (\ref{e:linplane}) with $\gamma=0$. Clearly the plane is graphical iff $|\alpha_1+\alpha_2|^2-|\beta_1+\beta_2|^2\neq0$.

For non-graphical planes, start with $|\alpha_1+\alpha_2|^2-|\beta_1+\beta_2|^2=0$, from which conclude the existence of $(\theta,\phi,H)\in[0,2\pi]\times[0,2\pi]\times{\mathbb R}$ such that
\[
\alpha_2=He^{i\theta}-\alpha_1 \qquad\qquad \beta_2=He^{i\phi}-\beta_1. 
\]
In fact, by combining equation (\ref{e:linplane}) and its complex conjugate, we can remove some redundancy in the description by setting $\alpha_1=1$ and $\beta_1=0$. Thus a non-graphical conformal plane is determined by the three real constants $(\theta,\phi,H)$.

Substituting equations (\ref{e:confco1}) and (\ref{e:confco2}) in equation (\ref{e:linplane}) with $\gamma=0$ and these simplifications yields a pair of real equations with solutions  (\ref{e:nongcp}).
\end{proof}
\vspace{0.1in}
This means in particular that a conformal plane is a section over a hemisphere and that it goes out to infinity along the fibre above the equator. 

\vspace{0.1in}
\begin{Prop}
The metric induced on a conformal plane has fixed signature throughout. For a graphical conformal plane the metric can be definite, indefinite or degenerate, while on a non-graphical plane  the induced metric can only be indefinite or degenerate.
\end{Prop}
\begin{proof}
Pulling back the metric (\ref{e:metric}) to the section (\ref{e:sect}) and taking the determinant yields
\[
{{Det}}\;{\mathbb G}|_\pi=\left(-(\alpha-\bar{\alpha})^2-4\beta\bar{\beta}\right)\frac{(1+\xi\bar{\xi})^2}{(1-\xi\bar{\xi})^2}.
\]
Thus the signature is fixed on graphical planes and one can get definite, indefinite or degenerate metrics according to whether $Im(\alpha)^2>|\beta|^2$, $Im(\alpha)^2<|\beta|^2$ or $Im(\alpha)^2=|\beta|^2$, respectively.

A similar calculation for non-graphical planes utilizing equation (\ref{e:nongcp}) leads to
\[
{{Det}}\;{\mathbb G}|_\pi=\frac{16H^2(e^{i\theta}-e^{-i\theta})^2}{(1+u^2)^2(1-u^4)^2}.
\]
The determinant is  zero or negative and so the metric is degenerate or indefinite according to whether $\theta\neq0,\pi$ or $\theta=0,\pi$, respectively.
\end{proof}
\vspace{0.1in}

Graphical indefinite conformal planes in $TS^2$ are totally real surfaces, having no complex points.  Here, complex is with respect to the canonical complex structure ${\mathbb J}$ on $TS^2$ \cite{kahlermetric}.  In fact, as there are no complex points at infinity ($|\xi|=1$) either, these twisting line congruences are examples of complete strictly totally real sections of $TS^2\rightarrow S^2$ that, were they Lagrangian with respect to the canonical symplectic structure $\omega$, could not exist by a conjecture of Toponogov \cite{gak2} \cite{Top}.  

Being symplectic, the 2-parameter family of lines in ${\mathbb R}^3$ defined by a graphical definite conformal plane in $TS^2$ twist around each other in such a way as to foliate ${\mathbb R}^3$. Conversely, foliations of ${\mathbb R}^3$ by lines are exactly the definite surfaces over a hemisphere in $TS^2$ \cite{salvai}.

Given a non-degenerate plane $\pi$ in ${\mathbb R}^{2,2}$, we defined its conjugate plane $\pi^\perp$ to be the plane orthogonal to $\pi$ with respect to the flat neutral metric. This definition of orthogonality is conformally invariant and thus, given a conformal plane in $TS^2$ we can define its conjugate conformal plane.
\vspace{0.1in}
\begin{Prop}
Two conformal planes are conjugate iff their tangent planes at the point of intersection are orthogonal with respect to ${\mathbb G}$. 

Two graphical conformal planes determined by equation (\ref{e:sect}) with $(\alpha,\beta)$ and $(\tilde{\alpha},\tilde{\beta})$ are conjugate iff $\tilde{\alpha}=\bar{\alpha}$ and $\tilde{\beta}=-\beta$. 

Two non-graphical conformal plane determined by equations (\ref{e:nongcp}) with $(\theta,\phi,H)$ and $(\tilde{\theta},\tilde{\phi},\tilde{H})$ are conjugate iff
\[
\tilde{\theta}=-\theta \qquad \tilde{\phi}=\phi \qquad \tilde{H}=-\frac{H}{1-2H\cos\theta}.
\]
\end{Prop}
\begin{proof}
A pair of planes $\pi,\tilde{\pi}$ in ${\mathbb R}^{2,2}$ intersecting at the origin are conjugate iff
\[
Z_{1}\bar{\tilde{Z}}_{1}-Z_{2}\bar{\tilde{Z}}_{2}=0,
\]
for all points with complex coordinates $(Z_{1},Z_{2})$ in $\pi$ and $(\tilde{Z}_{1},\tilde{Z}_{2})$ in $\tilde{\pi}$. 
If the plane is graphical, substituting equations (\ref{e:confco1}) and (\ref{e:sect}) in the last equation with the two sets of coefficients $(\alpha,\beta)$ and $(\tilde{\alpha},\tilde{\beta})$, one finds that $\tilde{\alpha}=\bar{\alpha}$ and $\tilde{\beta}=-\beta$. Similarly, substituting equations (\ref{e:confco1}) and (\ref{e:nongcp}), we get he stated relation for non-graphical planes. 

Consider the associated conformal planes in $TS^2$. When a conformal plane is graphical it is given by equation (\ref{e:sect}), and the tangent space at the origin is spanned by vectors of the form
\begin{align}
X&=A\left.\left(\frac{\partial}{\partial\xi}+\frac{\partial\eta}{\partial\xi}\frac{\partial}{\partial\eta}+\frac{\partial\bar{\eta}}{\partial\xi}\frac{\partial}{\partial\bar{\eta}}\right)\right|_0
+\bar{A}\left.\left(\frac{\partial}{\partial\bar{\xi}}+\frac{\partial{\eta}}{\partial\bar{\xi}}\frac{\partial}{\partial\eta}+\frac{\partial\bar{\eta}}{\partial\bar{\xi}}\frac{\partial}{\partial\bar{\eta}}\right)\right|_0\nonumber\\
\qquad &=A\left.\left(\frac{\partial}{\partial\xi}+\alpha\frac{\partial}{\partial\eta}+\bar{\beta}\frac{\partial}{\partial\bar{\eta}}\right)\right|_0
+\bar{A}\left.\left(\frac{\partial}{\partial\bar{\xi}}+\beta\frac{\partial}{\partial\eta}+\bar{\alpha}\frac{\partial}{\partial\bar{\eta}}\right)\right|_0,\nonumber
\end{align}
for $A\in{\mathbb C}$. A short computation shows that the tangent spaces of two graphical planes intersecting at the origin are orthogonal, ${\mathbb G}(X,\tilde{X})|_0=0$ for all $X\in T_0\pi$, $\tilde{X}\in T_0\tilde{\pi}$, iff $\tilde{\alpha}=\bar{\alpha}$ and $\tilde{\beta}=-\beta$. 

A similar calculation for non-graphical planes yields the stated result.

\end{proof}
\vspace{0.1in}
\subsection{Doubly ruled surfaces}

In \cite{fjohn} John identified points in ${\mathbb R}^{2,2}$ with lines in ${\mathbb R}^3$ and proved that the conjugate circles in Asgeirsson's theorem are the two rulings of the 1-sheeted hyperboloid of revolution. He then showed that maps in line space which come from affine transformations of $\mathbb R^3$ preserve solutions of the ultra-hyperbolic equation (\ref{UHE}). With this he extended the original Asgeirsson's theorem on the two families of generating lines of a hyperboloid of revolution to a mean value theorem on the rulings of a general $1-$sheeted hyperboloid. 

Given this relationship between the original Asgeirsson's Theorem and the 1-sheeted hyperboloid, is there a similar relationship between the extended version and the other doubly ruled surface, the hyperbolic paraboloid? 

The answer is yes, as we now prove.

\vspace{0.1in}
\noindent{\bf Theorem \ref{t:3}}

{\it
A pair of non-degenerate conjugate conics can be identified with a pair of rulings of the same surface in ${\mathbb{R}}^3$.

In the definite case, they are the two rulings of a 1-sheeted hyperboloid, in the hyperbolic case, they are either parts of the two rulings of a 1-sheeted hyperboloid or the two rulings of a hyperbolic paraboloid. 
}
\begin{proof}
The case of definite  (and hence graphical) planes is originally John's result - see Theorem 2.1 of \cite{fjohn}. 

The link between $TS^2$ and ${\mathbb R}^3$ is given by the map $\Phi:TS^2\times{\mathbb R}\rightarrow{\mathbb R}^3$ which takes an oriented line and a number, to the point in ${\mathbb R}^3$ on the line that lies the given distance from the point on the line closest to the origin. 

In coordinates $(\xi,\eta)$ on $TS^2$ and $(X_1,X_2,X_3)$ on ${\mathbb R}^3$, $\Phi((\xi,\eta),r)$ can be written explicitly
\begin{equation}\label{e:minit}
X_1+iX_2=\frac{2(\eta-\xi^2\bar{\eta})}{(1+\xi\bar{\xi})^2}+r\frac{\xi}{1+\xi\bar{\xi}}
\qquad
X_3=\frac{-2(\bar{\xi}\eta+\xi\bar{\eta})}{(1+\xi\bar{\xi})^2}+r\frac{1-\xi\bar{\xi}}{1+\xi\bar{\xi}}.
\end{equation}

For graphical conformal planes, our starting point is equation $(\ref{e:sect})$ which can be simplified by rotation and translation so that $\alpha=-ai$ and $\beta=b$ for $a,b\in{\mathbb R}$. To compute the unit (pseudo)-circle, note that the distance to the origin in flat coordinates via equation (\ref{e:confco1}) is
\[
Z_1\bar{Z}_1-Z_2\bar{Z}_2=\frac{16R^2(a+b\sin2\theta)}{(1-R^2)^2},
\]
where we have introduced polar coordinates $\xi=Re^{i\theta}$.

Here one sees again that the metric is definite on the plane if $a>b>0$ and so, in that case the unit circle in the conformal plane is given by
\begin{equation}\label{e:pcirc}
R=-\left[4(a+b\sin2\theta)\right]^{\scriptstyle{\frac{1}{2}}}+\left[4(a+b\sin2\theta)+1\right]^{\scriptstyle{\frac{1}{2}}}.
\end{equation}

Substituting equation (\ref{e:sect}) in (\ref{e:minit}) and restricting to the curve (\ref{e:pcirc}) one finds after a lengthy calculation that
 the ruled surface in ${\mathbb R}^3$ determined by the unit (pseudo)-circle in a graphical conformal plane satisfies 
\begin{equation}\label{e:ruling_gr}
X_3^2-4aX_1^2-4aX_2^2-8bX_1X_2+a^2-b^2=0,
\end{equation}
for the given constants $a,b\in{\mathbb R}$. This a one-sheeted hyperboloid, as claimed.

Moreover, the conjugate conic, obtained by switching the signs of $a$ and $b$, and the sign of the unit radius, satisfies exactly the same quadratic, as can be confirmed by the same calculation with these signs flipped.  Thus we have reproven John's result that a pair of conjugate circles in definite planes generate a pair of rulings of the hyperboloid.

When $a^2<b^2$, the metric is indefinite and consider the curve defined by equation $(\ref{e:pcirc})$ only for values of $\theta$ for which $R$ is real. These are the pseudo-circles of the Lorentz metric and the two rulings are of the same surface given by equation (\ref{e:ruling_gr}).

For a non-graphical conformal plane with parameters $(\theta,\phi,H)$ as given in equation (\ref{e:nongcp}), the distance to the origin in flat coordinates via equation (\ref{e:confco1}) is
\[
Z_1\bar{Z}_1-Z_2\bar{Z}_2=\frac{32uv\sin\theta}{(1+u^2)(1-u^2)^2}.
\]
Non-degeneracy of the induced metric means that $\theta\neq0,\pi$ and so one restricts to the unit pseudo-circle by setting
\[
v=\pm\frac{(1+u^2)(1-u^2)^2}{32u\sin\theta}.
\]
The resulting 1-parameter family of lines defines a ruled surface in ${\mathbb R}^3$ that, upon substitution in equations (\ref{e:minit}), is found to satisfy
\[
\sin\theta X_3+2(\cos\theta+\cos\phi)X_1^2+4\sin\phi X_1X_2+2(\cos\theta-\cos\phi)X_2^2=0.
\]
This is a hyperbolic paraboloid for $\theta\neq0,\pi$ and flipping to the conjugate pseudo-circle, we find it satisfies the same equation. This completes the proof.

\end{proof}
\vspace{0.1in}

\subsection{Conformal Asgeirsson's Theorem}

As we have seen, on ${\mathbb R}^{2,2}$ the flat Laplace operator $\Delta$ is the ultra-hyperbolic operator which acts on twice continuously differentiable functions $u:{\mathbb R}^{2,2}\rightarrow \mathbb R$ as
$$
\Delta u :=
\frac{\partial^2 u}{\partial x_1^2} + \frac{\partial^2 u}{\partial x_2^2} -
\frac{\partial^2 u}{\partial x_3^2} -
\frac{\partial^2 u}{\partial x_4^2}.
$$
Define $S_0,S_0^{\perp}\subset{\mathbb R}^{2,2}$ to be the pair of circles 
$$
S_0:=\{(\alpha_1,\alpha_2,0, 0) \in {\mathbb R}^{2,2} \ | \ \alpha_1^2 +\alpha_2^2 = 1\},
$$
and
$$
S_0^\perp:=\{(0,0,\beta_1,\beta_2) \in{\mathbb R}^{2,2} \ |\ \beta_1^2 +\beta_2^2 = 1\}.
$$
In this dimension the original mean value theorem of Asgeirsson can be stated:
\vspace{0.1in}

\begin{Thm} \cite{LA}\label{Asgeirsson}
Let $u:{\mathbb R}^{2,2}\rightarrow \mathbb R$ be a solution of $\Delta u=0$. Then the following integral equation holds
$$
\int_{S_0} u \ dl = \int_{S_0^\perp} u \ dl,
$$
where $dl$ represents the line element induced by the flat metric.
\end{Thm}

\vspace{0.1in}

The following extension of Asgeirsson's Theorem says that the mean value property is invariant under conformal transformations.

\vspace{0.1in}
\begin{Prop}
Let $u:{\mathbb R}^{2,2}\rightarrow \mathbb R$ be a solution of $\Delta u=0$ and $f:{\mathbb R}^{2,2} \rightarrow {\mathbb R}^{2,2}$ be a conformal map. Then

$$
\int_{f(S_0)} u \ dl = \int_{f(S_0^\perp)} u \ dl,
$$
where $dl$ represents the line element induced by the flat metric.
\end{Prop}
\vspace{0.1in}
\begin{proof}
Let $\Omega$ be the conformal factor of $f$.
Let $\alpha: I \rightarrow S_0$ be a parametrization of $S_0$. Then
\begin{align}
\int_{f(S_0)} u \ dl & = \int_I u((f\circ \alpha)(t)) \cdot  ||(f\circ\alpha)'(t)|| \ dt \\
& = \int_I (u\circ f)(\alpha(t)) \cdot  \Omega ||\alpha'(t)|| \ dt \\
& = \int_{S_0} \Omega(u\circ f) \ dl.
\end{align}

Similarly we would get

$$
\int_{f(S_0^\perp)} u \ dl = \int_{S_0^\perp} \Omega(u\circ f) \ dl.
$$

By Proposition \ref{conformalsolutions} we have that $\Delta \Omega (u\circ f)=0$, so we can apply Asgeirsson's theorem (Theorem \ref{Asgeirsson}) to the integrand $\Omega (u\circ f)$.
\end{proof}

\vspace{0.1in}

This can also be expressed more geometrically as

\vspace{0.1in}
\noindent{\bf Theorem \ref{t:4}}

{\it
Let $v:TS^2\rightarrow \mathbb R$ be a solution of $\Delta_{\mathbb G} v=0$ and $f:TS^2 \rightarrow TS^2$ be a conformal map. Then
\[
\int_{f(S_0)} v \ d\tau = \int_{f(S_0^\perp)} v \ d\tau.
\]
where $d\tau$ is the line element induced by the canonical metric ${\mathbb G}$.
}
\vspace{0.1in}

\section{Conformal Images of Conjugate Circles}\label{conformalimage}

\vspace{0.1in}
\subsection{Pseudo-Euclidean affine space}

Let ${\mathbb R}^{2,2}$ be regarded as an affine space, denoting the underlying vector space by $\overrightarrow{{\mathbb R}^{2,2}}$\footnote{In general, given an affine subspace $\sigma\subseteq {\mathbb R}^{2,2}$ we denote by $\overrightarrow \sigma \subset \overrightarrow{{\mathbb R}^{2,2}}$ its associated linear subspace, or {\em direction}.}.

Let $Q$ be the indefinite distance function associated with the metric, viewed as a quadratic form acting on vectors in $\overrightarrow{{\mathbb R}^{2,2}}$:
$$
Q(x) = x_1^2 + x_2^2 - x_3^2 -x_4^2 \quad x\in\overrightarrow{{\mathbb R}^{2,2}}.
$$
We use $||x||^2$ to mean $Q(x)$ and say that a vector $x\in \overrightarrow{{\mathbb R}^{2,2}}$ is {\em positive} if $||x||^2>0$, {\em negative} if $||x||^2<0$, and {\em null} whenever $||x||^2=0$. 

Let 
$$\langle x, y \rangle := x_1 y_1 + x_2 y_2 -x_3 y_3 - x_4 y_4,$$ denote the associate symmetric bilinear form to $Q$ for all $x,y\in\overrightarrow{{\mathbb R}^{2,2}}$.

\vspace{0.1in}
\begin{Def}
Two points $p,q\in {\mathbb R}^{2,2}$ are called \emph{skew} if $||\overrightarrow{pq}||^2 \neq 0$ and \emph{null-separated} if $||\overrightarrow{pq}||^2 = 0$. We will say that three points $q,q',q''\in {\mathbb R}^{2,2}$ are skew when they are pairwise skew.
\end{Def} 

We use the symbol $\perp$ to denote orthogonality with respect to $\langle, \rangle$. Given a linear subspace $W\subset \overrightarrow{{\mathbb R}^{2,2}}$ we define its orthogonal complement as usual
$$
W ^\perp :=\{v \in \overrightarrow{{\mathbb R}^{2,2}} \ | \ \langle v, w\rangle =0\ \text{ for all } w\in W \}.
$$

Let $P$ be a $2$-subspace of $\mathbb R^{2,2}$ through the origin. Define the {\em degeneracy} of the restriction $\langle, \rangle_P$ to be
$$\ker \langle, \rangle_P:=\{v\in P \subset \overrightarrow{{\mathbb R}^{2,2}} \ | \ \langle v, w\rangle = 0 \text{ for all } w\in P  \}.$$ 
\begin{Def}[Metric types of planes]\label{metrictypes}
A linear plane $P$ is said to be 

\begin{enumerate}
\item {\em non-degenerate} whenever $\ker \langle, \rangle_P = \{0\}$. In which case $P$ is called 

\begin{enumerate}
\item  {\em positive definite} when $Q_P\geq 0$, 
\item {\em negative definite} when $Q_P\leq 0$ and
\item {\em hyperbolic} otherwise. 
\end{enumerate}

\item {\em degenerate} or {\em parabolic} when $\ker \langle, \rangle_P$ has dimension $1$. We distinguish
\begin{enumerate} 
\item {\em positive parabolic} if $Q_P\geq 0$ and
\item {\em negative parabolic} if $Q_P\leq 0$.
\end{enumerate}
\item {\em totally null} when $\ker \langle, \rangle_P = P$. 
\end{enumerate}
\end{Def}

The same can be defined for an affine plane $\pi$. This way, $\pi$ will be non-degenerate iff its direction $\overrightarrow \pi$ is non-degenerate, $\pi$ will be positive definite iff $\overrightarrow \pi$ is positive definite, etc. Note that $\pi$ is the same as a conformal plane defined in Section \ref{s:2.3}. 

\vspace{0.1in}

\subsection{Non-degenerate conjugate conics}

\begin{Def}[Non-degenerate conjugate conics]
\label{conjugateconicsdef}
Let $O$ be a point in $\mathbb R^{2,2}$, $c$ a real number and $\pi$ a non-degenerate affine plane through $O$. The {\em conjugate conics} $S, S^\perp$ are 

$$
S:=\{p\in \pi \ | \ ||\overrightarrow{Op}||^2 =c^2\},
$$
\text{and}
$$
S^\perp:=\{p\in \pi^\perp \ | \ ||\overrightarrow{Op}||^2 =-c^2\},
$$
where $\pi^\perp:= O + {\overrightarrow \pi}^\perp$.

\end{Def}

\vspace{0.1in}

\begin{Note} \label{Asgcirclesareconjugate}
The circles $S_0$ and $S_0^\perp$ from Asgeirsson's theorem form a pair of non-degenerate conjugate conics. We can see this by letting $O$ be the origin of $\mathbb R^{2,2}$, $c=1$ and $\pi$ spanned by $(1,0,0,0)$ and $(0,1,0,0)$.
\end{Note}

In this section we want to show that $S^\perp$ can be obtained as the locus of points which lie $0$ distance (measured by $Q$) from any three points in $S$. Similarly, we will see that $S$ can be obtained as the locus of points lying $0-$distance from any three points in $S^\perp$. 

\begin{Def}
For $q\in {\mathbb R}^{2,2}$, define the {\em isotropic cone} $C_q$ to be
$$
C_q := \{ p \in {\mathbb R}^{2,2} \ | \ ||\overrightarrow{qp}||^2 = 0\}.
$$
Given 3 points $q,q',q''\in{\mathbb R}^{2,2}$, define

$$
N(q,q',q'') := C_q \cap C_{q'} \cap C_{q''}.
$$
\end{Def}

\vspace{0.1in}

\begin{Prop} \label{confnullity}
Let $f$ be a conformal map, then $f(C_p) = C_{f(p)}$ for all $p\in {\mathbb R}^{2,2}$.
\end{Prop}
\begin{proof} \cite[Theorem~3.32]{rosenfeld}.
\end{proof}

\vspace{0.1in}

\begin{Cor}\label{fN=Nf}
Let $f$ be a conformal map and $q,q',q''$ three points in ${\mathbb R}^{2,2}$. Then 
$$
f(N(q,q',q'')) = N(f(q), f(q'), f(q'')).
$$
\end{Cor}
\vspace{0.1in}

\begin{Prop}\label{duality} Let $S,S^\perp$ be a pair of non-degenerate conjugate conics. 
For all distinct $q,q',q''\in S$ 
$$
S^\perp = N(q,q',q''),
$$
and for all distinct $p,p',p''\in S^\perp$
$$
S = N(p,p',p'').
$$
\end{Prop}
\begin{proof}
The proof follows from two Lemmas:

\vspace{0.1in}

\begin{Lem}\label{conicthreepoints}
Let $q,q',q''$ be three skew points on a non-degenerate affine plane $\pi$. Then there exists a unique point $O\in \pi$ and a real number $c^2$ so that $||\overrightarrow{Oq}||^2 = ||\overrightarrow{Oq'}||^2 = ||\overrightarrow{Oq''}||^2 = c^2$.
\end{Lem}
\begin{proof} 
The locus of points in $\pi$ which are equidistant to $q$ and $q'$ is
$$m_{qq'}:=q + \frac{\overrightarrow{qq'}}{2} + {\mbox{span }}\{\overrightarrow{qq'}\}^\perp,$$

Similarly, let $m_{qq''}$ be the locus of points equidistant to $q$ and $q''$. Since
$q,q',q''$ are assumed to be in general position we have ${\mbox{span }}\{\overrightarrow{qq'}\}\neq {\mbox{span }}\{\overrightarrow{qq''}\}$. As $Q_{\overrightarrow{\pi}}$ is non-degenerate we will have  ${\mbox{span }}\{\overrightarrow{qq'}\}^\perp \neq {\mbox{span }}\{\overrightarrow{qq''}\}^\perp$. Hence the affine lines $m_{qq'}$ and $m_{qq''}$ must intersect on a unique point $O$ in $\pi$. Then $$||\overrightarrow{Oq}||^2 = ||\overrightarrow{Oq'}||^2 = ||\overrightarrow{Oq''}||^2,$$ and we define the radius $c^2$ to be its value.
\end{proof}

\vspace{0.1in}

\begin{Lem} \label{0to3}
Let $q,q',q''$ be three skew points in ${\mathbb R}^{2,2}$ on a non-degenerate plane $\pi$. Let $p\in N(q,q',q'')$, then
$$\langle \overrightarrow{pO}, \overrightarrow{\pi}\rangle=0.$$
\end{Lem}
\begin{proof}
Let $O\in \pi$ be the center of $q,q',q''$. Expanding $||\overrightarrow{pq}||^2=0$ and $||\overrightarrow{pq'}||^2=0$ we get
\begin{align*}
&||\overrightarrow{pO}||^2 +\langle \overrightarrow{pO}, \overrightarrow{Oq}\rangle +  ||\overrightarrow{Oq}||^2  = 0, \\
& ||\overrightarrow{pO}||^2 +\langle \overrightarrow{pO}, \overrightarrow{Oq'}\rangle +  ||\overrightarrow{Oq'}||^2  = 0.
\end{align*}
Substracting these equations we get $\langle \overrightarrow{pO}, \overrightarrow{qq'}\rangle = 0$, and similarly for $q''$. 
\end{proof}
\vspace{0.1in}

Returning to the proof of the Proposition, let $S,S^\perp$ be a pair of non-degenerate conjugate conics center $O$ and radius $c^2$. For $q,q',q''$ distinct points in $S$, we prove $S^\perp= N(q,q',q'')$. We first see $S^\perp\subseteq N(q,q',q'')$: Let $p\in S^\perp$. Observe 
$$||\overrightarrow{pq}||^2 = ||\overrightarrow{pO}||^2 + ||\overrightarrow{Oq}||^2 = -c^2 +c^2 =0,$$
and the same applies to any point in $S$, in particular $q'$ and $q''$. 

We now see $S^\perp\supseteq N(q,q',q'')$: Let $p\in N(q,q',q'')$. By Lemma \ref{0to3} $\langle \overrightarrow{pO}, \overrightarrow{Oq}\rangle = 0$, so
$$0 = ||\overrightarrow{pq}||^2 = ||\overrightarrow{pO}||^2 + ||\overrightarrow{Oq}||^2 = ||\overrightarrow{pO}||^2 +c^2 =0.$$

Hence $||\overrightarrow{pO}||^2 = -c^2$ and $p\in S^\perp$. \\

For $p,p',p''$ distinct points in $S^\perp$, we prove $S= N(p,p',p'')$.
$\subseteq$: Let $q\in S$, a similar argumentation as before works.

$\supseteq$: Let $q\in N(p,p',p'')$. By Lemma \ref{0to3} $\langle \overrightarrow{qO}, {\overrightarrow \pi}^\perp \rangle = 0$ so it follows that $q\in O + {{\overrightarrow \pi}^\perp}^\perp = \pi$.

This completes the proof of Proposition \ref{duality}.
\end{proof}

\vspace{0.1in}
\subsection{The conformal image of circles}

\begin{Prop} \label{3trans}
Conformal transformations act transitively on triples of skew points of ${\mathbb R}^{2,2}$.
\end{Prop}
\begin{proof}
For the proof of this Proposition we consider conformal transformations as being bijections of conformal space $C^4_2:={\mathbb R}^{2,2}\cup C_\infty$ (see \cite{russian} \cite{rosenfeld}).
Let $q,q',q''$ be three skew points in $C^4_2$. Let $e_i$ be the canonical basis of ${\mathbb R}^{2,2}$. We show that we can find a conformal mapping sending $q,q',q''$ to $\bf 0, \infty, 1$; where ${\bf 0}=(0,0,0,0)$ and ${\bf 1}=e_1$. \\

One can assume $q=\bf 0$ by use of a translation mapping to the origin. One can also assume $q'=\infty$ since an inversion with respect to the hypersphere centered at $q'$ going through $\bf 0$  maps $q'$ to $\infty$ while fixing $\bf 0$. We show that there is a conformal mapping fixing $\bf 0$ and $\infty$, which sends $q''$ to $\bf 1$. \\

The point $q''$ is non-null as $q$ and $q''$ were assumed to be skew, and conformal transformations preserve skewness of points. Assume $q''$ is positive and let $u_1:=q'', \ldots, u_4$ a basis of ${\mathbb R}^{2,2}$ with Gram matrix $I^{2,2}:={\mbox{diag }}(1,1,-1,-1)$. Let $M$ be the $4\times 4$ real matrix whose columns are the coordinates of the vectors $u_i$ with respect to the canonical basis $e_i$. Then $M^T I^{2,2} M =I^{2,2}$ and $Me_1= q''$, so $M$ is the matrix of an orthogonal transformation mapping $e_1$ to $q''$. Orthogonal transformations are conformal, and so we are done. \\

If $q''$ is negative then we can find a basis $v_1,v_2,v_3:= q'',v_4$ of ${\mathbb R}^{2,2}$ with Gram matrix $I^{2,2}$. By a similar argument as before we can find an orthogonal transformation of sending $e_3$ to $q''$. Precomposing with the {\em anti}-orthogonal (and therefore conformal) transformation $(a,b,c,d)\mapsto (c,d,a,b)$ we get a conformal map sending $q''$ to $e_1$.
\end{proof}

\vspace{0.1in}
\begin{Prop} 
Let $S,S^\perp$ be a pair of non-degenerate conjugate conics. There exists a conformal transformation $f: {\mathbb R}^{2,2} \rightarrow {\mathbb R}^{2,2}$ such that $S=f(S_0)$ and $S^\perp = f(S_0^\perp)$. 
\end{Prop}

\begin{proof}
Let $p,p',p''\in S^\perp$ and $p_0,p_0',p_0''\in S_0^\perp$. Any three distinct points on a non-degenerate pseudo-circle must be skew, so using the transitivity property we can find an $f$ such that $f(p_0)= p, f(p_0')=p'$ and $f(p_0'')=p''$. Then 
$$
f(S_0) = f(N(p_0,p_0',p_0'')) = N(p,p',p'')=S.
$$
Now choose $q_0,q_0',q_0''\in S$, and observe that $q:=f(q_0), q':=f(q_0'), q'':=f(q_0'')\in S$ as we have already seen $f(S_0) = S$. 

$$f(S_0^\perp) = f(N(q_0,q_0',q_0'')) = N(q,q',q'') = S^\perp.$$
\end{proof}
\vspace{0.1in}

Thus, every non-degenerate pair of conics can be mapped by a conformal map to the standard conjugate circles. By Theorem \ref{t:4} conformal maps preserve the mean value property and so the mean value property extends to any non-degenerate conics and we have completed the proof of Theorem \ref{t:1}.

\newpage
\appendix
\section{Examples}
Let 
$$
u(x_1,x_2, x_3,x_4) := \frac{\sqrt{10^6 p - 4(x_1+x_3)^2 - 4(x_2+x_4)^2 - (x_1^2 +x_2^2 - x_3^2 -x_4^2)^2}}{p}
$$
where $p=4 + (x_1 - x_3)^2 + (x_2 - x_4)^2$. Then $u$ is a solution of 

$$
\frac{\partial^2 u}{\partial x_1^2} + \frac{\partial^2 u}{\partial x_2^2} -
\frac{\partial^2 u}{\partial x_3^2} -
\frac{\partial^2 u}{\partial x_4^2} = 0.
$$

\includegraphics[scale = .5]{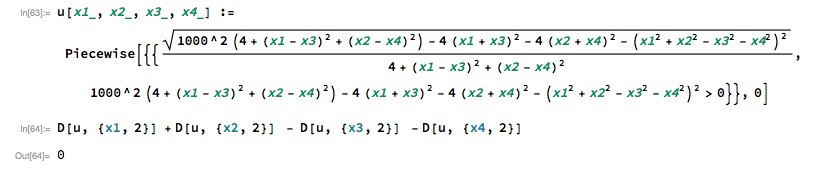}

We consider three skew points $q,q',q''\in \mathbb R^{2,2}$ and we want to see that the line integral of $u$ over $S(q,q',q'')$ is equal to that over $S^\perp(q,q',q'')$. Let $\pi$ be the affine plane defined by points $q,q',q''$ and let $P$ be its direction.

\subsection{Definite case}

Take the skew points $q=(8,0,0,0), q'=(7,1,0,0), q''=(6,0,0,0)$ with direction $P= {\mbox{span }}\{(1,0,0,0), (0,1,0,0)\}$. $P$ is non-degenerate so these 3 points have a center in $\pi$. The center is $$O=(7,0,0,0),$$ 
and the radius is $c^2 = 1$. We can compute $P^\perp = {\mbox{span }}\{(0,0,1,0), (0,0,0,1)\}$. 
$$
S(q,q',q'') = \{ (7 + \cos t, \sin t, 0,0) \ | \ t\in [0, 2\pi)\}
$$
$$
S^\perp(q,q',q'') = \{ (7,0,\cos t, \sin t) \ | \ t\in [0, 2\pi) \}
$$

\begin{center}\includegraphics[scale = .5]{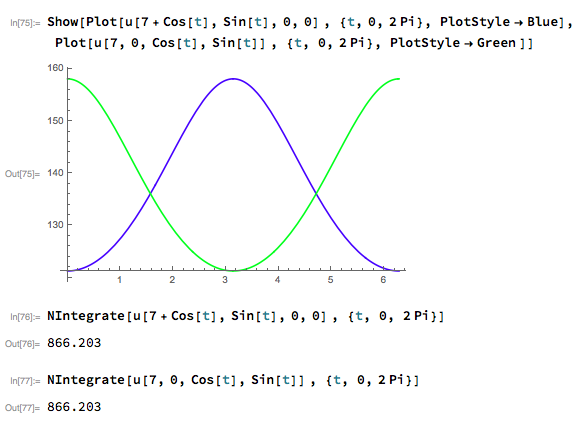}
\end{center}

\subsection{Hyperbolic case}
Take the skew points $q=(8,0,0,0), q'=(6,0,0,0), q''=(9,0,\sqrt{3},0)$ for which $P= {\mbox{span }}\{(1,0,0,0), (0,0,1,0)\}$. $P$ is non-degenerate therefore $q,q',q''$ have a center in $\pi$. The center is $O = (7,0,0,0)$ and the radius is $c^2=1$. We compute $P^\perp = {\mbox{span }}\{(0,1,0,0), (0,0,0,1)\}$. 
$$
S(q,q',q'') = \{(7\pm \cosh t, 0, \sinh t,0) \ | \ t\in \mathbb R \}
$$
$$
S^\perp(q,q',q'') = \{ (7,\sinh t,0, \pm \cosh t)\ | \ t\in \mathbb R \}
$$
\begin{center}
\includegraphics[scale = .5]{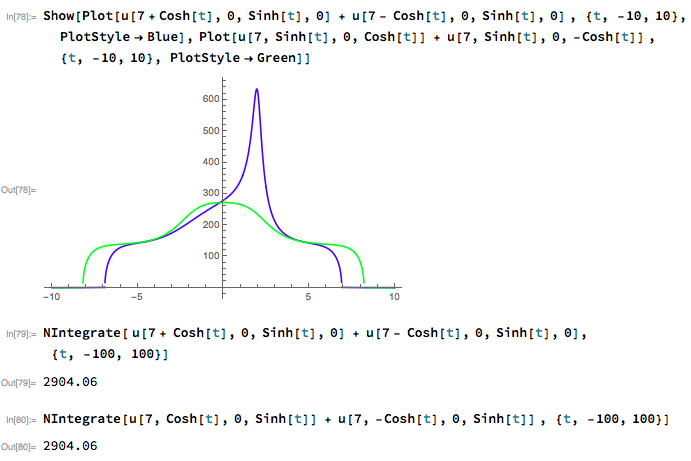}
\end{center}

\newpage

\section{Some Solutions of the UHE}

We now look at three examples of solutions of the UHE which arise as integrals of constant functions with simple support: a thickened plane, a 3-ball and $k$ separate 3-balls.

\vspace{0.1in}
\noindent{\bf Example 1}:

Consider the function $f:{\mathbb R}^3\rightarrow {\mathbb R}$ given by
\[
f(x^1,x^2,x^3)=\left\{\begin{array}{ll}
  1   &  {{for }}\;|x^3|\leq d_0\\
  0   & {{otherwise}}
\end{array}\right.
\]
where $d_0$ is a strictly positive real number.

\vspace{0.1in}
\begin{Prop}
The integral over lines of the function $f$ is a function $u:TS^2\rightarrow{\mathbb R}$ given by
\[
u(\xi,\eta)=\int_{(\xi,\eta)} f dl=\left\{\begin{array}{ll}
  \frac{1+\xi\bar{\xi}}{1-\xi\bar{\xi}}d_0   &  {{for }}\;|\xi|\neq1\\
  0   & {{for }}\;|\xi|=1\;{{ and }}\;2|\eta|>(1+\xi\bar{\xi})d_0\\
  \infty   & {{for }}\;|\xi|=1\;{{ and }}\;2|\eta|\leq(1+\xi\bar{\xi})d_0
\end{array}\right.
\]
\end{Prop}
\vspace{0.1in}

Note that the function $u$ is infinite for some lines as the support of $f$ is not compact. Also, $u$ is the metric conformal factor, which satisfies the ultra-hyperbolic equation.

\vspace{0.1in}
\noindent{\bf Example 2}:

Consider the function $f:{\mathbb R}^3\rightarrow {\mathbb R}$ given by
\[
f(x^1,x^2,x^3)=\left\{\begin{array}{ll}
  1   &  {{for }}\;(x^1)^2+(x^2)^2+(x^3)^2\leq r^2_0\\
  0   & {{otherwise}}
\end{array}\right.
\]
where $r_0$ is a strictly positive real number.
\vspace{0.1in}
\begin{Prop}
The integral over lines of the function $f$ is
\[
u(\xi,\eta)=\left\{\begin{array}{ll}
  2\sqrt{r_0^2-\frac{4\eta\bar{\eta}}{(1+\xi\bar{\xi})^2}}   &  {{for }}\;2|\eta|\leq(1+\xi\bar{\xi})r_0\\
  0   & {{for }}\;2|\eta|>(1+\xi\bar{\xi})r_0
\end{array}\right.
\]
\end{Prop}

\vspace{0.1in}
\noindent{\bf Example 3}:

The previous example can be extended to any constant function $f:{\mathbb R}^3\rightarrow {\mathbb R}$ with support given by $k$ disjoint 3-balls of radii $\{r_j\}_{j=1}^k$ centres at $\{p_j\}_{j=1}^k$.

That is, suppose that
\[
f(x^1,x^2,x^3)=\left\{\begin{array}{ll}
  d_j   &  {{for }}\;(x^1-x^1_j)^2+(x^2-x^2_j)^2+(x^3-x^3_j)^2\leq r^2_j\;,\;1\leq j\leq k\\
  0   & {{otherwise}}.
\end{array}\right.
\]
For each point $p_j$ define the section of $TS^2\rightarrow S^2$
\[
\eta_j=\frac{1}{2}(x^1_j+ix^2_j-2x^3_j\xi-(x^1_j-ix^2_j)\xi^2)
\]
and the functions $\delta_j:TS^2\rightarrow\{0,1\}$
\[
\delta_j(\xi,\eta)=\left\{\begin{array}{ll}
  1   &  {{for }}\;4|\eta-\eta_j|^2\leq (1+\xi\bar{\xi})^2r_j\\
  0   & {{for }}\;4|\eta-\eta_j|^2> (1+\xi\bar{\xi})^2r_j.
\end{array}\right.
\]
\vspace{0.1in}
\begin{Prop}
The integral over lines of the function $f$ with support $k$ disjoint 3-balls is a linear superposition of the individual solutions of the UHE:
\[
u(\xi,\eta)=\sum_{j=1}^kd_j\delta_j\left[r_j^2-\frac{4|\eta-\eta_j|^2}{(1+\xi\bar{\xi})^2}\right]^{\scriptstyle{\frac{1}{2}}}
\]
\end{Prop}

\vspace{0.1in}

We can therefore interpret the linearity of the ultra-hyperbolic equation as the non-interaction of densities at different points.

\vspace{0.1in}
\noindent{\bf Acknowledgement}: During the development of this work the first author was supported by an Institute of Technology Tralee Postgraduate Research Scholarship.

\end{document}